\documentclass{amsart}

\usepackage{amscd,amssymb}

\newcommand{\I}{\mathbf 1}

\newcommand{\N}{\mathbf N}
\newcommand{\Q}{\mathbf Q}

\newcommand{\Z}{\mathbf Z}

\newcommand{\bG}{\mathbf G}

\newcommand{\sA}{\mathcal A}

\newcommand{\sC}{\mathcal C}

\newcommand{\sJ}{\mathcal J}

\newcommand{\sL}{\mathcal L}
\newcommand{\sM}{\mathcal M}

\newcommand{\sP}{\mathcal P}

\newcommand{\sV}{\mathcal V}

\newcommand{\iso}{\xrightarrow{\sim}}

\newcommand{\nd}{\nobreakdash-\hspace{0pt}}

\renewcommand{\theenumi}{(\roman{enumi})}

\DeclareMathOperator*{\colim}{colim}
\DeclareMathOperator{\End}{End}

\DeclareMathOperator{\Hom}{Hom}

\DeclareMathOperator{\pr}{pr}

\DeclareMathOperator{\REP}{REP}
\DeclareMathOperator{\Spec}{Spec}

\DeclareMathOperator{\tr}{tr}

\newtheorem{thm}{Theorem}[section]

\newtheorem{lem}[thm]{Lemma}
\newtheorem{prop}[thm]{Proposition}

\theoremstyle{definition}

\newtheorem*{thm*}{Theorem}
\newtheorem*{defn*}{Definition}

\numberwithin{equation}{section}

\begin{document}

\title{A finiteness theorem for algebraic cycles}

\author{Peter O'Sullivan}
\address{Centre for Mathematics and its Applications\\
Australian National University, Canberra ACT 0200 \\
Australia}
\email{peter.osullivan@anu.edu.au}
\thanks{The author was partly supported by ARC Discovery Project grant DP0774133.}



\date{}

\dedicatory{}

\begin{abstract}
Let $X$ be a smooth projective variety.
Starting with a finite set of cycles on powers $X^m$ of $X$, we consider the
$\Q$\nd vector subspaces of the $\Q$\nd linear Chow groups of the $X^m$
obtained by iterating the algebraic operations
and pullback and push forward along those morphisms $X^l \to X^m$ for which
each component $X^l \to X$ is a projection.
It is shown that these $\Q$\nd vector subspaces are all finite-dimensional, provided
that the $\Q$\nd linear Chow motive of $X$ is a direct summand of that of an abelian variety.
\end{abstract}

\maketitle

\section{Introduction}

Let $X$ be a  smooth projective variety over a field $F$.
Starting with a finite set of cycles on powers $X^m$ of $X$,
consider the $\Q$\nd vector subspaces $C_m$ of the $\Q$\nd linear Chow groups $CH(X^m)_\Q$
formed by iterating the
algebraic operations and pullback $p^*$ and push forward $p_*$ along those morphisms
$p:X^l \to X^m$ for which each component $X^l \to X$ is a projection.
It is plausible that the $C_m$ are always finite-dimensional,
because when $F$ is finitely generated over the prime field it is plausible that the
$CH(X^m)_\Q$ themselves are finite-dimensional.
In this paper we prove the finite-dimensionality of the $C_m$
when the $\Q$\nd linear Chow motive of $X$ is a direct summand of that of an abelian
variety over $F$.

More precisely, suppose given an adequate
equivalence relation $\sim$ on $\Q$\nd linear cycles on smooth projective
varieties over $F$.
We say that $X$ is a \emph{Kimura variety for $\sim$} if, in the
category of $\Q$\nd linear Chow motives modulo $\sim$, the motive of
$X$ is the direct sum of one for which some exterior power is $0$ and
one for which some symmetric power is $0$.
A Kimura variety for $\sim$ is also a Kimura variety for any coarser equivalence relation.
It is known (e.g.\ \cite{Kim}, Corollary~4.4) that if the $\Q$\nd linear Chow motive of $X$ is a direct summand of that of
an abelian variety, then $X$ is a Kimura variety for any $\sim$.
The main result is now Theorem~\ref{t:fin} below.
In addition to a finiteness result, it contains also a nilpotence result.
By a filtration $C^\bullet$ on a graded $\Q$\nd algebra $C$ we mean a descending
sequence $C = C^0 \supset C^1 \supset \dots$ of graded ideals of $C$
such that $C^r.C^s \subset C^{r+s}$ for every $r$ and $s$.
The morphisms $p:X^l \to X^m$ in Theorem~\ref{t:fin}~\ref{i:pullpush} are
exactly those defined by maps $[1,m] \to [1,l]$.

\begin{thm}\label{t:fin}
Let $X$ be a smooth projective variety over $F$ which is a Kimura variety
for the equivalence relation ${\sim}$.
For $n = 0,1,2, \dots $, let $Z_n$ be a finite subset of $CH(X^n)_\Q/{\sim}$ , with $Z_n$ empty
for $n$ large.
Then there exists for each $n$ a graded $\Q$\nd subalgebra
$C_n$ of $CH(X^n)_\Q/{\sim}$, and a filtration $(C_n)^\bullet$ on $C_n$,
with the following properties.
\begin{enumerate}
\renewcommand{\theenumi}{(\alph{enumi})}
\item\label{i:pullpush}
If $p:X^l \to X^m$ is a morphism for which each component $X^l \to X$ is a projection,
then $p^*$ sends $C_m$ to $C_l$ and $p_*$ sends $C_l$ to $C_m$, and $p^*$ and $p_*$
respect the filtrations on $C_l$ and $C_m$.
\item\label{i:fin}
For every $n$, the $\Q$\nd algebra $C_n$ is finite-dimensional and contains $Z_n$.
\item\label{i:nilp}
For every $n$, the cycles in $C_n$ which are numerically equivalent to $0$ lie in $(C_n)^1$,
and we have $(C_n)^r = 0$ for $r$ large.
\end{enumerate}
\end{thm}

The finiteness result of Theorem~\ref{t:fin} is non-trivial only
if $Z_n$ is non-empty for some $n > 1$.
Indeed it will follow from Proposition~\ref{p:Chowsub} below that
for any smooth projective variety $X$ over $F$ and finite subset $Z_1$ of $CH(X)_\Q$,
there is a finite-dimensional graded $\Q$\nd subalgebra $C_n$ of $CH(X^n)_\Q$ for each $n$ such that
$C_1$ contains $Z_1$ and $p^*$ sends $C_m$ to $C_l$ and $p_*$ sends $C_l$ to $C_m$
for every $p$ as in \ref{i:pullpush} of Theorem~\ref{t:fin}.

If $X$ is a Kimura variety for $\sim$, then the ideal of correspondences
numerically equivalent to $0$ in the algebra $CH(X \times_F X)_\Q/{\sim}$
of self-correspondences of $X$ has been shown by Kimura (\cite{Kim}, Proposition~7.5)
to consist of nilpotent elements, and by Andr\'e and Kahn (\cite{AndKah}, Proposition~9.1.14)
to be in fact nilpotent.
The nilpotence result of Theorem~\ref{t:fin} implies that of Kimura, but neither
implies nor is implied by that of Andr\'e and Kahn.

If $\sim$ is numerical equivalence, then the $CH(X^m)_\Q/{\sim} = \overline{CH}(X^m)_\Q$
are all finite dimensional.
The following result shows that they are generated in a suitable sense by
the $\overline{CH}(X^m)_\Q$ for $m$ not exceeding some fixed $n$.

\begin{thm}\label{t:num}
Let $X$ be a smooth projective variety over $F$ which is a Kimura variety
for numerical equivalence.
Then there exists an integer $n \ge 0$ with the following property:
for every $m$, the $\Q$\nd vector space $\overline{CH}(X^m)_\Q$
is generated by elements of the form
\begin{equation}\label{e:numgen}
p_*((p_1)^*(z_1).(p_2)^*(z_2). \cdots .(p_r)^*(z_r)),
\end{equation}
where $z_i$ lies in $\overline{CH}(X^{m_i})_\Q$ with $m_i \le n$,
and $p:X^l \to X^m$ and the $p_i:X^l \to X^{m_i}$
are morphisms for which each component $X^l \to X$ is a projection.
\end{thm}

Theorem~\ref{t:fin} will be proved in Section~\ref{s:finproof} and
Theorem~\ref{t:num} in Section~\ref{s:numproof}.
Both theorems are deduced from the following fact.
Given a Kimura variety $X$ for $\sim$, there is a reductive group $G$
over $\Q$, a finite-dimensional $G$\nd module $E$,
a central $\Q$\nd point $\rho$ of $G$ with $\rho^2 = 1$, and a commutative
algebra $R$ in the tensor category $\REP(G,\rho)$ of $G$\nd modules
with symmetry modified by $\rho$, with the following property.
If we write $\sM_{\sim}(F)$ for the category of ungraded motives over $F$ modulo ${\sim}$ and
$E_R$ for the free $R$\nd module $R \otimes_k E$ on $E$, then
there exist isomorphisms
\[
\xi_{r,s}:\Hom_{G,R}((E_R)^{\otimes r},(E_R)^{\otimes s}) \iso
\Hom_{\sM_{\sim}(F)}(h(X)^{\otimes r},h(X)^{\otimes s})
\]
which are compatible with composition and tensor products of morphisms and
with symmetries interchanging the factors $E_R$ and $h(X)$.
These isomorphisms arise because there is a fully faithful tensor functor
from the category of finitely generated free $R$\nd modules to $\sM_{\sim}(F)$,
which sends $E_R$ to $h(X)$ (see \cite{O}, Lemma~3.3 for a similar result).
However to keep the exposition as brief as possible, the $\xi_{r,s}$ will
simply be constructed directly here, in Sections~\ref{s:Kimobj} and \ref{s:Kimvar}.
Now we have an equality
\[
CH(X^r)_\Q/{\sim} = \Hom_{\sM_{\sim}(F)}(\I,h(X)^{\otimes r})
\]
of $\Q$\nd algebras, and pullback along a morphism $f:X^l \to X^m$ is given by composition
with $h(f)$.
There is a canonical autoduality $h(X)^{\otimes 2} \to \I$ on $h(X)$,
and push forward along $f$ is given by composing with the transpose of $h(f)$ defined
by this autoduality.
Using the isomorphisms $\xi_{r,s}$, Theorems~\ref{t:fin} and \ref{t:num} then reduce to easily
solved problems about the $G$\nd algebra $R$.

\section{Generated spaces of cycles}

The following result gives an explicit description of the spaces of cycles
generated by a given set of cycles on the powers of an arbitrary
smooth projective variety $X$.
By the top Chern class of $X$ we mean the element $(\Delta_X)^*(\Delta_X)_*(1)$
of $CH(X)_\Q$, where $\Delta_X:X \to X^2$ is the diagonal.
We define the tensor product $z \otimes z'$ of $z$ in $CH(X)_\Q$
and $z'$ in $CH(X')_\Q$ as $\pr_1{}\!^*(z).\pr_2{}\!^*(z')$ in
$CH(X \times_F X')_\Q$.
The push forward of $z \otimes z'$ along a morphism $f \times f'$ is then
$f_*(z) \otimes f'{}_*(z')$.

\begin{prop}\label{p:Chowsub}
Let $X$ be a smooth projective variety over $F$.
For $m = 0,1,2, \dots$ let $Z_m$ be a subset of $CH(X^m)_\Q$,
such that $Z_1$ contains the top Chern class of $X$.
Denote by $C_m$ the $\Q$\nd vector subspace of $CH(X^m)_\Q$
generated by elements of the form
\begin{equation}\label{e:pullpush}
p_*((p_1)^*(z_1).(p_2)^*(z_2). \cdots .(p_n)^*(z_r)),
\end{equation}
where $z_i$ lies in $Z_{m_i}$,
and $p:X^j \to X^m$ and the $p_i:X^j \to X^{m_i}$
are morphisms for which each component $X^j \to X$ is a projection.
Then $C_m$ is a $\Q$\nd subalgebra of $CH(X^m)_\Q$ for each $m$.
If $q:X^l \to X^m$ is a morphism for which each component $X^l \to X$
is a projection, then $q^*$ sends $C_m$ into $C_l$ and
$q_*$ sends $C_l$ into $C_m$.
\end{prop}

\begin{proof}
Write $\sP_{l,m}$ for the set of morphisms $X^l \to X^m$ for which each
component $X^l \to X$ is a projection.
Then the composite of an element of $\sP_{j,l}$ with an element of $\sP_{l,m}$
lies in $\sP_{j,m}$.
Thus $q_*(C_l) \subset C_m$ for $q$ in $\sP_{l,m}$.
Similarly the product of two elements of $\sP_{j,m}$ lies in
$\sP_{2j,2m}$.
Thus the tensor product of two elements of $C_m$ lies in $C_{2m}$.
Since the product of two elements of $CH(X^m)_\Q$ is the pullback
of their tensor product along $\Delta_{X^m} \in \sP_{m,2m}$,
it remains only to show that
\begin{equation}\label{e:qC}
q^*(C_m) \subset C_l
\end{equation}
for $q$ in $\sP_{l,m}$.
This is clear when $l = m$ and $q$ is a symmetry $\sigma$
permuting the factors $X$ of $X^m$,
because $\sigma^* = (\sigma^{-1})_*$.
An arbitrary $q$ factors for some $l'$ as $q'' \circ q'$ with $q'$ in
$\sP_{l,l'}$ a projection and $q''$ in  $\sP_{l',m}$ a closed immersion.
It thus is enough to show that \eqref{e:qC} holds when $q$ is a projection
or $q$ is a closed immersion.

Suppose that $q$ is a projection.
If we write
\[
w_{s,n}:X^{s+n} \to X^n
\]
for the projection onto the last $n$ factors,
then $q$ is a composite of a symmetry with $w_{l-m,m}$.
Thus \eqref{e:qC} holds because $w_{l-m,m}{}^* = 1 \otimes -$ sends $C_m$ into $C_l$.

Suppose that $q$ is a closed immersion.
Then $q$ is a composite of the
\[
e_s = X^{s-2} \times \Delta_X:X^{s-1} \to X^s
\]
for $s \ge 2$ and  symmetries.
To prove \eqref{e:qC}, we may thus suppose that
$m \ge 2$ and $q = e_m$.

Denote by $W_m$ the $\Q$\nd subalgebra of $CH(X^m)_\Q$ generated by the
$v^*(Z_{m'})$ for any $m'$ and $v$ in $\sP_{m,m'}$.
Then $u^*$ sends $W_m$ into $W_l$ for any $u$ in $\sP_{l,m}$, and
by the projection formula $C_m$ is a $W_m$\nd submodule of $CH(X^m)_\Q$.
Since $C_m$ is the sum of the $p_*(W_j)$ with $p$ in $\sP_{j,m}$,
it is to be shown that
\begin{equation}\label{e:emp}
(e_m)^*p_*(W_j) \subset C_{m-1}
\end{equation}
for every $p$ in $\sP_{j,m}$.
We have $p = w_{j,m} \circ \Gamma_p$ with $\Gamma_p$ the graph of $p$, and
\[
(e_m)^* \circ (w_{j,m})_* = (w_{j,m-1})_* \circ (e_{j+m})^*.
\]
Thus \eqref{e:emp} will hold provided that $(e_{j+m})^*(\Gamma_p)_*(W_j) \subset C_{j+m-1}$.
Replacing $m$ by $j+m$ and $p$ by $\Gamma_p$, we may thus suppose that
$p$ has a left inverse in $\sP_{m,j}$.
In that case any $y$ in $W_j$ is of the form $p^*(x)$ with $x$ in $W_m$,
and then
\[
(e_m)^*p_*(y) = (e_m)^*(p_*(1).x) = (e_m)^*p_*(1).(e_m)^*(x).
\]
Thus \eqref{e:emp} will hold provided that $(e_m)^*p_*(1)$ lies in $C_{m-1}$.

To see that $(e_m)^*p_*(1)$ lies in $C_{m-1}$, note that $e_m$ has a left inverse
$f$ in $\sP_{m,m-1}$.
Then
\[
(e_m)^*p_*(1) = f_*(e_m)_*(e_m)^*p_*(1) = f_*(p_*(1).(e_m)_*(1)) = f_*p_*p^*(e_m)_*(1).
\]
Since $(e_m)_*(1) = (w_{m-2,2})^*(\Delta_X)_*(1)$, we reduce finally to showing that
$h^*(\Delta_X)_*(1)$ lies in $C_j$ for every $h$ in $\sP_{j,2}$.
Such an $h$ factors as a projection followed by either a symmetry of $X^2$ or $\Delta_X$,
so we may suppose that $j = 1$ and $h = \Delta_X$.
Then $h^*(\Delta_X)_*(1)$ is the top Chern class of $X$, which by hypothesis lies
in $Z_1 \subset C_1$.
\end{proof}

\section{Group representations}\label{s:grprep}

Let $k$ be a field.
By a $k$\nd linear category we mean a category equipped with a structure
of $k$\nd vector space on each hom-set such that the composition is $k$\nd bilinear.
A $k$\nd linear category is said to be pseudo-abelian if it has a zero object
and direct sums, and if every idempotent endomorphism has an image.
A \emph{$k$\nd tensor category} is a pseudo-abelian $k$\nd linear category $\sC$,
together with a structure of symmetric monoidal category on $\sC$ (\cite{Mac},~VII~7)
such that the tensor product $\otimes$ is $k$\nd bilinear on hom-spaces.
Thus $\sC$ is equipped with a unit object $\I$, and natural isomorphisms
\[
(L \otimes M) \otimes N \iso L \otimes (M \otimes N),
\]
the associativities,
\[
M \otimes N \iso N \otimes M,
\]
the symmetries, and $\I \otimes M \iso M$ and $M \otimes \I \iso M$, which satisfy appropriate
compatibilities.
We assume in what follows that $\I \otimes M \iso M$ and $M \otimes \I \iso M$ are
identities: this can be always arranged by replacing if necessary $\otimes$ by an isomorphic
functor.
Brackets in multiple tensor products will often be omitted when it is of no importance
which bracketing is chosen.
The tensor product of $r$ copies of $M$ will then be written as $M^{\otimes r}$,
and similarly for morphisms.
There is a canonical action $\tau \mapsto M^{\otimes \tau}$ of the symmetric group
$\mathfrak{S}_r$ on $M^{\otimes r}$,
defined using the symmetries.
It extends to a homomorphism of $k$\nd algebras from $k[\mathfrak{S}_r]$
to $\End(M^{\otimes r})$.
When $k$ is of characteristic $0$, the symmetrising idempotent in $k[\mathfrak{S}_r]$
maps to an idempotent endomorphism $e$ of $M^{\otimes r}$, and we define the $r$th
symmetric power $S^r M$ of $M$ as the image of $e$.
Similarly we define the $r$th exterior power $\bigwedge^r M$ of $M$ using the
antisymmetrising idempotent in $k[\mathfrak{S}_r]$

Let $G$ be a linear algebraic group over $k$.
We write $\REP(G)$ for the category of $G$\nd modules.
The tensor product $\otimes_k$ over $k$ defines on $\REP(G)$
a structure of $k$\nd tensor category.
Recall (\cite{Wat}, 3.3) that every $G$\nd module is the filtered colimit of its finite-dimensional
$G$\nd submodules.
If $E$ is a finite-dimensional $G$\nd module, then regarding $\REP(G)$
as a category of comodules (\cite{Wat}, 3.2) shows that $\Hom_G(E,-)$ preserves filtered
colimits.
When $k$ is algebraically closed, a $k$\nd vector subspace of a $G$\nd module is a
$G$\nd submodule provided it is stable under every $k$\nd point of $G$.
This is easily seen by reducing to the finite-dimensional case.

We suppose from now on that $k$ has characteristic $0$.
Let $\rho$ be a central $k$\nd point of $G$ with $\rho^2 = 1$.
Then $\rho$ induces a $\Z/2$\nd grading on $\REP(G)$, with the
$G$\nd modules pure of degree $i$ those on which $\rho$ acts as $(-1)^i$.
We define as follows a $k$\nd tensor category $\REP(G,\rho)$.
The underlying $k$\nd linear category, tensor product and associativities of $\REP(G,\rho)$
are the same as those of $\REP(G)$, but the symmetry  $M \otimes N \iso N \otimes M$ is given
by multiplying that in $\REP(G)$ by $(-1)^{ij}$ when $M$ is of degree $i$ and $N$
of degree $j$, and then extending by linearity.
When $\rho = 1$, the $k$\nd tensor categories $\REP(G)$ and $\REP(G,\rho)$ coincide.

An algebra in a $k$\nd tensor category is defined as an object $R$
together with a multiplication $R \otimes R \to R$ and unit $\I \to R$ satisfying
the usual associativity and identity conditions.
Since the symmetry is not used in this definition, an algebra in $\REP(G,\rho)$ is
the same as an algebra in $\REP(G)$, or equivalently a $G$\nd algebra.
An algebra $R$ in $\REP(G,\rho)$  will be said to be finitely generated if its underlying
$k$\nd algebra is.
It is equivalent to require that $R$ be generated as a $k$\nd algebra
by a finite-dimensional $G$\nd submodule.

A (left) module over an algebra $R$ is an object $N$ equipped with an action
$R \otimes N \to N$ satisfying the usual associativity and identity conditions.
If $R$ is an algebra in $\REP(G,\rho)$ or $\REP(G)$, we also speak of a $(G,R)$\nd module.
A $(G,R)$\nd module is said to be finitely
generated if it is finitely generated as a module
over the underlying $k$\nd algebra of $R$.
It is equivalent to require that it be generated as a module over the $k$\nd algebra $R$ by a
finite-dimensional $G$\nd submodule.

An algebra $R$ in a $k$\nd tensor category is said to be commutative if
composition with the symmetry interchanging the factors $R$ in $R \otimes R$
leaves the multiplication unchanged.
If $R$ is an algebra in $\REP(G,\rho)$,
this notion of commutativity does not in general coincide with that of the
underlying $k$\nd algebra, but it does
when $\rho$ acts as $1$ on $R$.

Coproducts exist in the category of commutative algebras in a $k$\nd tensor category:
the coproduct of $R$ and $R'$ is $R \otimes R'$ with multiplication the tensor
product of the multiplications of $R$ and $R'$ composed with the appropriate
symmetry.
To any map $[1,m] \to [1,l]$ and commutative algebra $R$ there is then associated a
morphism $R^{\otimes m} \to R^{\otimes l}$, defined using symmetries $R^{\otimes \tau}$ and
the unit and multiplication of $R$ and their tensor products and composites,
such that each component $R \to  R^{\otimes l}$ is the embedding into one of the factors.

Let $R$ be a commutative algebra in $\REP(G,\rho)$.
Then the symmetry in $\REP(G,\rho)$ defines on any $R$\nd module a canonical structure
of $(R,R)$\nd bimodule.
The category of $(G,R)$\nd modules
has a structure of $k$\nd tensor category, with the tensor product $N \otimes_R N'$
of $N$ and $N'$ defined in the usual way as the coequaliser of the two
morphisms
\[
N \otimes_k R \otimes_k N' \to N \otimes_k N'
\]
given by the actions of $R$ on $N$ and $N'$, and the tensor product $f \otimes_R f'$
of $f:M \to N$ and $f':M' \to N'$ as the unique morphism rendering the square
\[
\begin{CD}
M \otimes_R M' @>{f \otimes_R f'}>> N \otimes_R N' \\
@AAA                                      @AAA     \\
M \otimes_k M'  @>{f \otimes_k f'}>> N \otimes_k N'
\end{CD}
\]
commutative.

Let $P$ be an object in $\REP(G,\rho)$.
We write $P_R$ for the object $R \otimes_k P$ in the $k$\nd tensor category of $(G,R)$\nd modules.
A morphism of commutative algebras $R' \to R$ in $\REP(G,\rho)$ induces by tensoring
with $P$ a morphism of $R'$\nd modules $P_{R'} \to P_R$.
For each $l$ and $m$, extension of scalars along $R' \to R$
then gives a $k$\nd linear map
\begin{equation}\label{e:extscal}
\Hom_{G,R'}((P_{R'})^{\otimes m},(P_{R'})^{\otimes l}) \to
\Hom_{G,R}((P_R)^{\otimes m},(P_R)^{\otimes l})
\end{equation}
Explicitly, \eqref{e:extscal} sends $f'$ to the unique morphism of $(G,R)$\nd modules $f$
that renders the square
\[
\begin{CD}
(P_R)^{\otimes m} @>{f}>>  (P_R)^{\otimes l}  \\
@AAA                                               @AAA              \\
(P_{R'})^{\otimes m}    @>{f'}>>                    (P_{R'})^{\otimes l}
\end{CD}
\]
commutative, where the vertical arrows are those defined by $P_{R'} \to P_R$.
If $P$ is finite-dimensional, then for given commutative algebra $R$
and $f$,
there is a finitely generated $G$\nd subalgebra $R'$ of $R$ such that $f$ is
in the image of \eqref{e:extscal}.
This can be seen by writing $R$ as the filtered colimit $\colim_\lambda R_\lambda$
of its finitely generated $G$\nd subalgebras, and noting that since $P^{\otimes m}$ is finite-dimensional,
the composite of $P^{\otimes m} \to (P_R)^{\otimes m}$ with $f$ factors through some
$(P_{R_\lambda})^{\otimes l}$.

Suppose that $G$ is reductive, or equivalently that $\Hom_G(P,-)$ is exact for every
$G$\nd module $P$.
Then $\Hom_G(P,-)$ preserves colimits for $P$ finite-dimensional.
In particular $(-)^G = \Hom_G(k,-)$ preserves colimits.
If $R$ is a commutative algebra in $\REP(G,\rho)$ with $R^G = k$,
then $R$ has a unique maximal $G$\nd ideal.
Indeed $J^G = 0$ for $J \ne R$ a $G$\nd ideal of $R$,
while $(J_1)^G = 0$ and $(J_2)^G = 0$ implies $(J_1 + J_2)^G = 0$.

\begin{lem}\label{l:repfin}
Let $G$ be a reductive group over a field $k$ of characteristic $0$ and $\rho$ be a
central $k$\nd point of $G$ with $\rho^2 = 1$.
Let $R$ be a finitely generated
commutative algebra in $\REP(G,\rho)$ with $R^G = k$, and $N$ be a
finitely generated $R$\nd module.
\begin{enumerate}
\item\label{i:algfin}
The $k$\nd vector space $N^G$ is finite-dimensional.
\item\label{i:idealcompl}
For every $G$\nd ideal $J \ne R$ of $R$,
we have
$(J^rN)^G = 0$ for $r$ large.
\end{enumerate}
\end{lem}

\begin{proof}
Every object $P$ of $\REP(G,\rho)$ decomposes as $P_0 \oplus P_1$
where $\rho$ acts as $(-1)^i$ on $P_i$.
In particular $R = R_0 \oplus R_1$ with $R_0$ a $G$\nd subalgebra of $R$.
Suppose that $R$ is generated as an algebra by the finite-dimensional $G$\nd submodule $M$.
Then $R_0$ is generated as an algebra by $M_0 + M_1{}\!^2$, and hence is finitely generated.
Since $R$ is a commutative algebra in $\REP(G,\rho)$, it is generated as an $R_0$\nd module
by $M_1$.
Hence any finitely generated $R$\nd module is finitely generated
as an $R_0$\nd module.

To prove \ref{i:algfin}, we reduce after replacing $R$ by $R_0$
to the case where $R = R_0$.
Then $R$ is a commutative $G$\nd algebra in the usual sense.
In this case it is well known that $N^G$ is finite-dimensional over $k = R^G$
(e.g \cite{ShaAlgIV}, II~Theorem~3.25).

To prove \ref{i:idealcompl}, note that $J_0 \ne R_0$ is an ideal of $R_0$.
Since $R$ is a finitely generated $R_0$\nd module,
so also is $R_1$.
If $x_1, x_2, \dots ,x_s$ generate $R_1$ over $R_0$, then since each $x_i$ has square $0$
we have $R_1{}\!^r = 0$ and hence $J_1{}\!^r = 0$ for $r>s$.
Thus for $r > s$ we have
\[
J^r N
= (J_0 + J_1)^r N
= J_0{}\!^r N + J_0{}\!^{r-1}J_1 N + \dots + J_0{}\!^{r-s}J_1{}\!^s N
\subset J_0{}\!^{r-s} N.
\]
Replacing $R$ by $R_0$ and $J$ by $J_0$,
we thus reduce again to the case where $R = R_0$ is a commutative $G$\nd algebra
in the usual sense.
We may suppose further that $k$ is algebraically closed.

By \ref{i:algfin}, it is enough to show that $\bigcap_{r=0}^\infty J^r N = 0$,
or equivalently
(\cite{BAC-1} III \S 3 No.~2 Proposition~5 and IV \S 1 No.~1 Proposition~2, Corollaire~2)
that $J + \mathfrak{p} \ne R$ for every associated prime
$\mathfrak{p}$ of $N$.
Fix such a $\mathfrak{p}$, and consider the intersection $\mathfrak{p}'$
of the $g\mathfrak{p}$ for $g$ in $G(k)$.
It is stable under $G(k)$, and hence since $k$ is algebraically closed is a $G$\nd ideal of $R$.
Thus $J +  \mathfrak{p}' \ne R$,
because $J$ and $\mathfrak{p}'$ are contained in the unique maximal $G$\nd ideal of $R$.
Since each $g\mathfrak{p}$ lies in the finite set of associated primes of $N$,
it follows that $J + g\mathfrak{p} \ne R$ for some $g$ in $G(k)$.
Applying $g^{-1}$ then shows that $J + \mathfrak{p} \ne R$.
\end{proof}

Let $l_0$ and $l_1$ be integers $\ge 0$.
Write
\begin{equation}\label{e:Gdef}
G = GL_{l_0} \times_k GL_{l_1},
\end{equation}
$E_i$ for the standard representation of $GL_{l_i}$, regarded as a $G$\nd module, and
\begin{equation}\label{e:Edef}
E = E_0 \oplus E_1.
\end{equation}
We may identify the endomorphism of $E$ that sends $E_i$ to itself and acts on it as $(-1)^i$
with a central $k$\nd point $\rho$ of $G$ with $\rho^2 = 1$.

Consider the semidirect product
\begin{equation}\label{e:semidir}
\Gamma_r = (\Z/2)^r \rtimes \mathfrak{S}_r,
\end{equation}
where the symmetric group $\mathfrak{S}_r$ acts on $(\Z/2)^r$ through its action
on $[1,r]$.
For each $r$, the group $\Gamma_r$ acts on $E^{\otimes r}$,
with the action of
$(\Z/2)^r$ the tensor product of the actions $i \mapsto \rho^i$ of $\Z/2$ on $E$,
and the action of $\mathfrak{S}_r$ that defined by the
symmetries in $\REP(G,\rho)$.
Thus we obtain a homomorphism
\begin{equation}\label{e:semidirhom}
k[\Gamma_r] \to \End_G(E^{\otimes r})
\end{equation}
of $k$\nd algebras.

For $r \le r'$ we may regard $\Gamma_r$ as a subgroup of $\Gamma_{r'}$, and hence
$k[\Gamma_r]$ as a $k$\nd subalgebra of $k[\Gamma_{r'}]$, by identifying
$(\Z/2)^r$ with the subgroup of $(\Z/2)^{r'}$ with the last $r'-r$ factors the identity
and $\mathfrak{S}_r$ with the subgroup of $\mathfrak{S}_{r'}$ which leaves
the last $r'-r$ elements of $[1,r']$ fixed.
Write $e_0$ for the idempotent of $k[\Z/2]$ given by half the sum of the two elements
of $\Z/2$, and $e_1$ for $1-e_0$.
Given $\pi = (\pi_1,\pi_2,\dots,\pi_r)$ in $(\Z/2)^r$, we then have an idempotent
\[
e_\pi = e_{\pi_1} \otimes e_{\pi_2} \otimes \dots \otimes e_{\pi_r}
\]
in $k[\Z/2]^{\otimes r} = k[(\Z/2)^r] \subset k[\Gamma_r]$.
When every component of $\pi$ is $i \in \Z/2$,
we write $e_{i,r}$ for $e_\pi$.
We also write $a_{0,r}$ for the antisymmetrising idempotent and
$a_{1,r}$ for the symmetrising idempotent in $k[\mathfrak{S}_r]$, and for $i \in \Z/2$ we write
\begin{equation}\label{e:xir}
x_{i,r} = e_{i,l_i +1} a_{i,l_i +1} = e_{i,l_i +1} a_{i,l_i +1} e_{i,l_i +1}
= a_{i,l_i +1} e_{i,l_i +1} \in  k[\Gamma_{l_i +1}] \subset k[\Gamma_r]
\end{equation}
if  $r > l_i$ and $x_{i,r} = 0$ otherwise.

\begin{lem}\label{l:GL}
\begin{enumerate}
\item\label{i:hom0}
If $r \ne r'$ then $\Hom_G(E^{\otimes r},E^{\otimes r'}) = 0$.
\item\label{i:homsurj}
The homomorphism \eqref{e:semidirhom} is surjective, with kernel the ideal of $k[\Gamma_r]$
generated by $x_{0,r}$ and $x_{1,r}$.
\end{enumerate}
\end{lem}

\begin{proof}
\ref{i:hom0}
The action of $G$ on $E$ restricts along the appropriate $\bG_m \to G$ to
the homothetic action of $\bG_m$ on $E$.

\ref{i:homsurj}
Write $I$ for the ideal of
$k[\Gamma_r]$ generated by $x_{0,r}$ and $x_{1,r}$.
The image of $e_\pi$ under \eqref{e:semidirhom} is
the projection onto the direct summand
\[
E_\pi  = E_{\pi_1} \otimes_k E_{\pi_2} \otimes_k \dots \otimes_k E_{\pi_r}
\]
of $E^{\otimes r}$.
The $e_\pi$ give a decomposition of the identity of $k[\Gamma_r]$ into orthogonal idempotents,
and \eqref{e:semidirhom} is the direct sum over $\pi$ and $\pi'$ of the homomorphisms
\begin{equation}\label{e:semidirhompi}
e_{\pi'} k[\Gamma_r] e_\pi \to \Hom_G(E_\pi,E_{\pi'})
\end{equation}
it induces on direct summands of $k[\Gamma_r]$ and $\End_G(E^{\otimes r})$.
It is thus enough to show that \eqref{e:semidirhompi} is surjective, with
kernel $e_{\pi'} I e_\pi$.

Restricting to the centre
of $G$ shows that the target of \eqref{e:semidirhompi} is $0$ unless
$\pi'$ and $\pi$ have the same number of components $0$ or $1$,
or equivalently $\pi' = \tau \pi \tau^{-1}$ for some $\tau \in \mathfrak{S}_r$.
The same holds for the source of \eqref{e:semidirhompi}, because
\[
\tau e_\pi \tau^{-1} = e_{\tau \pi \tau^{-1}}
\]
for every $\tau$ and $\pi$.
Since further the image of $\tau \in \mathfrak{S}_r$ under \eqref{e:semidirhom} induces an
isomorphism from $E_\pi$ to $E_{\tau \pi \tau^{-1}}$,
to show that \eqref{e:semidirhompi} has the required properties we may suppose
that $\pi' = \pi$ and that $r = r_0 + r_1$ where the first $r_0$ components of
$\pi$ are $0$ and the last $r_1$ are $1$.
Then the source of \eqref{e:semidirhompi} has a basis
$e_\pi \tau e_\pi = e_\pi \tau$
with $\tau$ in the subgroup $\mathfrak{S}_{r_0} \times \mathfrak{S}_{r_1}$ of
$\mathfrak{S}_r$ that permutes the first $r_0$ and last $r_1$ elements of $[1,r]$
among themselves.
Thus we may identify
\[
k[\mathfrak{S}_{r_0}] \otimes_k k[\mathfrak{S}_{r_1}]
= k[\mathfrak{S}_{r_0} \times \mathfrak{S}_{r_1}]
\]
with the (non-unitary) $k$\nd subalgebra  $e_\pi k[\Gamma_r] e_\pi$ of $k[\Gamma_r]$.
Similarly we may identify
\[
\End_G(E_0{}\!^{\otimes r_0}) \otimes_k \End_G(E_1{}\!^{\otimes r_1})
= \End_G(E_0{}\!^{\otimes r_0} \otimes_k E_1{}\!^{\otimes r_1})
\]
with the (non-unitary) $k$\nd subalgebra $\End_G(E_\pi)$ of $\End_G(E)$.
Now given $\tau$ and $\tau'$ in $\mathfrak{S}_r$, the element $e_\pi \tau' x_{i,r} \tau^{-1} e_\pi$
is $0$ unless both $\tau$ and $\tau'$ send $[1,l_i + 1]$ into $[1,r_0]$ if $i = 0$
or into $[r_0+1,r]$ if $i = 1$.
With the above identifications, $e_\pi I e_\pi$ is thus the ideal of
$k[\mathfrak{S}_{r_0}] \otimes_k k[\mathfrak{S}_{r_1}]$ generated by $y_0 \otimes 1$
and $1 \otimes y_1$, where $y_i$ is $a_{i,l_i+1}$ in
$k[\mathfrak{S}_{l_i +1}] \subset k[\mathfrak{S}_{r_i}]$
if $r_i > l_i$ and $y_i = 0$ otherwise.
Further \eqref{e:semidirhompi} is the tensor product of
the homomorphisms
\begin{equation}\label{e:symhomi}
k[\mathfrak{S}_{r_i}] \to \End_G(E_i{}\!^{\otimes r_i})
\end{equation}
of $k$\nd algebras sending $\tau \in \mathfrak{S}_{r_i}$ to
$E_i{}\!^{\otimes \tau}$ in $\REP(G,\rho)$.
It will thus suffice to prove that \eqref{e:symhomi} is surjective with kernel
generated by $y_i$.
If $i = 0$, \eqref{e:symhomi} may be identified with the homomorphism defined
by the action of $\mathfrak{S}_{r_i}$ by symmetries on the $r_i$th tensor
power in $\REP(GL_{l_i})$ of the standard representation of $GL_{l_i}$,
while if $i = 1$, the composite of the automorphism
$\tau \mapsto \mathrm{sgn}(\tau) \tau$
of $k[\mathfrak{S}_{r_i}]$ with \eqref{e:symhomi} may be so identified.
The required result is thus classical (e.g. \cite{FulHar}, Theorem~6.3).
\end{proof}

\section{Duals}\label{s:dual}

Let $\sC$ be a $k$\nd tensor category.
By a duality pairing in $\sC$ we mean a quadruple $(L,L^\vee,\eta,\varepsilon)$
consisting of objects $L$ and $L^\vee$ of $\sC$ and morphisms
$\eta:\I \to L^\vee \otimes L$, the unit, and
$\varepsilon:L \otimes L^\vee \to \I$, the counit,
satisfying triangular identities analogous for those of an adjunction
(\cite{Mac}, p.~85).
Explicitly, it is required that, modulo associativities,
the composite of $L \otimes \eta$ with $\varepsilon \otimes L$ should be $1_L$, and of
$\eta \otimes L^\vee$ with $L^\vee \otimes \varepsilon$ should be $1_{L^\vee}$.
When such an $(L,L^\vee,\eta,\varepsilon)$ exists for a given $L$, it is said to be a duality pairing
for $L$, and $L$ is said to be dualisable, and $L^\vee$ to be dual to $L$.
We then have a dual pairing $(L^\vee,L,\widetilde{\eta},\widetilde{\varepsilon})$
for $L^\vee$, with $\widetilde{\eta}$ and $\widetilde{\varepsilon}$
obtained from $\eta$ and $\varepsilon$ by composing with the appropriate symmetries.

In verifying the properties of duals recalled below, it is useful to reduce to the
case where $\sC$ is strict, i.e.\ where all
associativities of $\sC$ are identities.
This can be done by taking (see \cite{Mac},~XI~3, Theorem~1) a $k$\nd linear 
strong symmetric monoidal functor (\cite{Mac},~XI~2) $\sC \to \sC'$ giving an equivalence
to a strict $k$\nd tensor category $\sC'$.

Suppose given duality pairings $(L,L^\vee,\eta,\varepsilon)$ for $L$
and $(L',L'{}^\vee,\eta',\varepsilon')$ for $L'$.
Then we have a tensor product duality pairing for $L \otimes L'$, with dual
$L^\vee \otimes L'{}^\vee$, and unit and counit
obtained from $\eta \otimes \eta'$ and $\varepsilon \otimes \varepsilon'$ by
composing with the appropriate symmetries.
Further any morphism $f:L \to L'$ has a transpose $f^\vee:L'{}^\vee \to L^\vee$,
characterised by the condition
\[
\varepsilon \circ (L \otimes f^\vee) = \varepsilon' \circ (f \otimes L'{}^\vee),
\]
or by a similar condition using $\eta$ and $\eta'$.
Explicitly, $f^\vee$ is given modulo associativities by the composite
of $\eta \otimes L'{}^\vee$ with $L^\vee \otimes f \otimes L'{}^\vee$ and
$L^\vee \otimes \varepsilon'$.
We have $(1_L)^\vee = 1_{L^\vee}$ and $(f' \circ f)^\vee = f^\vee \circ f'{}^\vee$,
and, with the transpose of $f^\vee$ taken using the dual pairing,
we have $f^{\vee \vee} = f$.
In particular taking $L = L'$ shows that a duality pairing for $L$ is unique up
to unique isomorphism.

Let $L$ be a dualisable object of $\sC$.
Then we have a $k$\nd linear map
\[
\tr_L:\Hom_\sC(N \otimes L,N' \otimes L) \iso \Hom_\sC(N,N'),
\]
natural in $N$ and $N'$,
which sends $f$ to its contraction $\tr_L(f)$ with respect to $L$,
defined as follows.
Modulo associativities, $\tr_L(f)$
is the composite of $N \otimes \widetilde{\eta}$
with $f \otimes L^\vee$ and $N' \otimes \varepsilon$, with $L^\vee$ and $\varepsilon$
as above
and $\widetilde{\eta}$ the composite of $\eta$ with the symmetry interchanging
$L^\vee$ and $L$.
It does not depend on the choice of duality pairing for $L$.
When $N = N' = \I$, the contraction $\tr_L(f)$ is the trace $\tr(f)$ of
the endomorphism $f$ of $L$.
The rank of $L$ is defined as $\tr(1_L)$.
Modulo associativities, $\tr_{L \otimes L'}$ is given by
successive contraction with respect to $L'$ and $L$, and $\tr_L$ commutes
with $M \otimes -$.
By the appropriate triangular identity for $L$ we have
\begin{equation}\label{e:gcomp}
g'' \circ g' = \tr_L((g'' \otimes g') \circ \sigma)
\end{equation}
for $g':M' \to L$ and $g'':L \to M''$,
with $\sigma$ the symmetry interchanging $M'$ and $L$.

Let $L$ be a dualisable object of $\sC$, and let $\tau$ be a permutation of $[1,r+1]$
and $f_1,f_2, \dots ,f_{r+1}$ be endomorphisms of $L$.
Write $\tau'$ for the permutation of $[1,r]$ obtained by omitting $r+1$ from the cycle
of $\tau$ containing it,
and define endomorphisms $c$ of $\I$ and $f'{}\!_1,f'{}\!_2, \dots ,f'{}\!_r$ of $L$ as follows.
If $\tau$ leaves $r+1$ fixed, then $c = \tr(f_{r+1})$ and $f'{}\!_i = f_i$ for $i\le r$.
If $\tau$ sends $r+1$ to $i_0 \le r$, then $c = 1$, and $f'{}\!_i$ for $i\le r$ is  $f_i$
when $i \ne i_0$ and $f_{i_0} \circ f_{r+1}$ when $i = i_0$.
We then have
\begin{equation}\label{e:symcontr}
\tr_L((f_1 \otimes f_2 \otimes \dots \otimes f_{r+1}) \circ L^{\otimes \tau}) =
c ((f'{}\!_1 \otimes f'{}\!_2 \otimes \dots \otimes f'{}\!_r) \circ L^{\otimes \tau'}).
\end{equation}
To see this, reduce to the case where $\tau$ leaves all but the last two elements
of $[1,r+1]$ fixed, by composing on the left and right with appropriate morphisms
$L^{\otimes \tau_0} \otimes L$ with $\tau_0$ a permutation of $[1,r]$.

Let $L$, $L'$, $M$ and $M'$ be objects in $\sC$, and $(L,L^\vee,\eta,\varepsilon)$
and $(L',L'{}^\vee,\eta',\varepsilon')$ be duality pairings for $L$ and $L'$.
Then we have a canonical isomorphism
\[
\Hom_\sC(M,M' \otimes L) \iso \Hom_\sC(M \otimes L^\vee,M')
\]
which modulo associativities sends $f:M \to M'\otimes L$ to the composite of
$f \otimes L^\vee$ with $M' \otimes \varepsilon$.
Its inverse is defined using $\eta$.
We also have a canonical isomorphism
\[
\Hom_\sC(L' \otimes M,M') \iso \Hom_\sC(M,L'{}^\vee \otimes M')
\]
defined using $\eta'$.
Replacing $M$ by $L' \otimes M$ in the first of these isomorphisms
and by $M \otimes L^\vee$ in the second,
and using the symmetries interchanging $M$ and $L'$ and $L'{}^\vee$ and $M'$,
then gives a canonical isomorphism
\[
\delta_{M,L;M',L'}:
\Hom_\sC(M \otimes L',M' \otimes L) \iso
\Hom_\sC(M \otimes L^\vee,M' \otimes L'{}^\vee).
\]
Modulo associativities, $\delta_{M,L;M',L'}$ sends $f$ to the composite of
$M \otimes \widetilde{\eta}' \otimes L^\vee$,
the tensor product of $f$ with the symmetry interchanging $L'{}^\vee$ and $L^\vee$,
and $M' \otimes \varepsilon \otimes L'{}^\vee$, where $\widetilde{\eta}'$
is $\eta'$ composed with the symmetry interchanging $L'{}^\vee$ and $L'$.

With the transpose taken using the chosen duality pairings for $L$ and $L'$, we have
\begin{equation}\label{e:deltahg}
\delta_{M,L;M',L'}(h \otimes g) = h \otimes g^\vee.
\end{equation}
With the duality pairing $(\I,\I,1_\I,1_\I)$ for $\I$, we have
\begin{equation}\label{e:deltaf}
\delta_{M,L;\I,\I}(f) = \varepsilon \circ (f \otimes L^\vee).
\end{equation}
With the tensor product duality pairings for $L_1 \otimes L_2$ and $L_1{}\!' \otimes L_2{}\!'$,
we have
\begin{multline}\label{e:deltatens}
\sigma''' \circ
(\delta_{M_1,L_1;M_1{}\!',L_1{}\!'}(f_1) \otimes \delta_{M_2,L_2;M_2{}\!',L_2{}\!'}(f_2))
\circ \sigma''
=  \\
= \delta_{M_1 \otimes M_2,L_1 \otimes L_2;M_1{}\!' \otimes M_2{}\!',L_1{}\!' \otimes L_2{}\!'}
(\sigma' \circ (f_1 \otimes f_2) \circ \sigma)
\end{multline}
where each of $\sigma$, $\sigma'$, $\sigma''$ and $\sigma'''$
is a symmetry interchanging the middle two factors in a tensor
product $(- \otimes -) \otimes (- \otimes -)$.

If $M'$ is dualisable, we have
\begin{equation}\label{e:deltaMcomp}
\delta_{M',L';M'',L''}(f') \circ \delta_{M,L;M',L'}(f) =
\delta_{M,L;M'',L''}(\tr_{M' \otimes L'}
(\sigma_2 \circ (f' \otimes f) \circ \sigma_1))
\end{equation}
for $\sigma_1$ the symmetry interchanging $M$ and $M'$ and $\sigma_2$ the symmetry
interchanging $L'$ and $L$.
This can be seen by showing
that modulo associativities both sides of
\eqref{e:deltaMcomp} coincide with a morphism obtained from
\[
f' \otimes f \otimes L^\vee \otimes L''{}^\vee \otimes M'{}^\vee \otimes L'{}^\vee
\]
as follows: compose on the left and right with appropriate symmetries, then on the left
with the tensor product of $M'' \otimes L''{}^\vee$ and the counits for $L$, $L'$ and $M'$
and on the right with the tensor product of $M \otimes L^\vee$ with the units
for $L''$, $L'$ and $M'$.
To show this in the case of the left hand side of \eqref{e:deltaMcomp},
write it as a contraction
with respect to $M' \otimes L'{}^\vee$ using \eqref{e:gcomp}
and contract first with respect to $L'{}^\vee$, using the triangular identity.

With the duality pairing $(\I,\I,1_\I,1_\I)$ for $\I$ and the tensor product
duality pairing for $L \otimes N$,
we have
\begin{equation}\label{e:deltaNcomp}
\delta_{M,N;\I,\I}(g \circ f) = (\delta_{M,L;\I,\I}(f) \otimes \delta_{L,N;\I,\I}(g))
\circ \sigma \circ  \delta_{M,N;M \otimes L,L \otimes N}(\alpha),
\end{equation}
with $\alpha:M \otimes (L \otimes N) \iso (M \otimes L) \otimes N$ the associativity
and $\sigma$ the symmetry interchanging $L$ and $L^\vee$ in the tensor product of
$M \otimes L$ and $L^\vee \otimes N^\vee$.
Indeed modulo associativities $\sigma \circ  \delta_{M,N;M \otimes L,L \otimes N}(\alpha)$
is $1_M \otimes \eta \otimes 1_{N^\vee}$ by the
triangular identity for $N$, and \eqref{e:deltaNcomp} then follows by the triangular
identity for $L$.

Let $(L,L^\vee,\eta,\varepsilon)$ be a duality pairing for
the object $L$ of $\sC$.
Its $r$th tensor power $(L^{\otimes r},(L^\vee)^{\otimes r},\eta_r,\varepsilon_r)$
is a duality pairing for $L^{\otimes r}$.
We write
\[
L^{r,s} = L^{\otimes r} \otimes (L^\vee)^{\otimes s}.
\]
Then $L^{r,0} = L^{\otimes r}$.
We define a $k$\nd bilinear product $\widetilde{\otimes}$ on morphisms between the $L^{r,s}$
by requiring that the square
\begin{equation}\label{e:tildedef}
\begin{CD}
L^{r_1,s_1} \otimes L^{r_2,s_2}       @>{\sim}>>     L^{r_1+r_2,s_1+s_2} \\
@V{f_1 \otimes f_2}VV        @VV{f_1 \mathbin{\widetilde{\otimes}} f_2}V \\
L^{r_1{}\!',s_1{}\!'} \otimes L^{r_2{}\!',s_2{}\!'} @>{\sim}>>
                                   L^{r_1{}\!'+r_2{}\!',s_1{}\!'+s_2{}\!'}
\end{CD}
\end{equation}
commute, with the top isomorphism the symmetry interchanging the two factors
$(L^\vee)^{\otimes s_1}$ and $L^{\otimes r_2}$ and the bottom
that interchanging $(L^\vee)^{\otimes s'{}\!_1}$ and $L^{\otimes r'{}\!_2}$.
Then $\widetilde{\otimes}$ preserves composites, is associative, and we have
\begin{equation}\label{e:tildecom}
f_2 \mathbin{\widetilde{\otimes}} f_1 =
\sigma' \circ (f_1 \mathbin{\widetilde{\otimes}} f_2) \circ \sigma^{-1},
\end{equation}
where $\sigma$ interchanges the first $r_1$ with the last $r_2$ factors $L$
and the first $s_1$ with the last $s_2$ factors $L^\vee$ of $ L^{r_1+r_2,s_1+s_2}$,
and similarly for $\sigma'$.
We define an isomorphism
\begin{equation}\label{e:deltaLdef}
\delta_{L;r,s;r',s'}:
\Hom_\sC(L^{\otimes (r+s')},L^{\otimes (r'+s)}) \iso \Hom_\sC(L^{r,s},L^{r',s'})
\end{equation}
by taking $L^{\otimes r},L^{\otimes s},L^{\otimes r'},L^{\otimes s'}$
for $M,L,M',L'$ in $\delta_{M,L;M',L'}$.

It follows from \eqref{e:deltahg} that
\begin{equation}\label{e:deltaLhg}
\delta_{L;r,s;r',s'}(h \otimes g) = h \otimes g^\vee.
\end{equation}
and from \eqref{e:deltaf} that
\begin{equation}\label{e:deltaLf}
\delta_{L;r,s;0,0}(f) = \varepsilon_s \circ (f \otimes (L^\vee)^{\otimes s}).
\end{equation}
By \eqref{e:deltatens}, we have
\begin{multline}\label{e:deltaLtens}
\delta_{L;r_1,s_1;r_1{}\!',s_1{}\!'}(f_1) \mathbin{\widetilde{\otimes}}
\delta_{L;r_2,s_2;r_2{}\!',s_2{}\!'}(f_2) = \\
= \delta_{L;r_1 + r_2,s_1 + s_2;r_1{}\!' + r_2{}\!',s_1{}\!' + s_2{}\!'}
(\sigma' \circ (f_1 \otimes f_2) \circ \sigma)
\end{multline}
for appropriate symmetries $\sigma$ and $\sigma'$.
By \eqref{e:deltaMcomp}, we have
\begin{equation}\label{e:deltaLcomp}
\delta_{L;r',s';r'',s''}(f') \circ \delta_{L;r,s;r',s'}(f) =
\delta_{L;r,s;r'',s''}(\tr_{L^{\otimes (r'+s')}}
(\sigma_2 \circ (f' \otimes f) \circ \sigma_1)).
\end{equation}
for appropriate symmetries $\sigma_1$ and $\sigma_2$.
We have
\begin{equation}\label{e:deltacomp}
\delta_{L;r,t;0,0}(g \circ f)  =
(\delta_{L;r,s;0,0}(f) \mathbin{\widetilde{\otimes}} \delta_{L;s,t;0,0}(g)) \circ
\delta_{L;r,t;r+s,s+t}(1_{L^{\otimes (r+s+t)}}),
\end{equation}
by \eqref{e:deltaNcomp}.

Let $G$ be a linear algebraic group over $k$ and $\rho$ be a central $k$\nd point
of $G$ with $\rho^2 = 1$.
Let $E$ be a finite-dimensional $G$\nd module and $R$ be a commutative algebra in $\REP(G,\rho)$.
Then $E$ in $\REP(G,\rho)$ and $E_R$ in the $k$\nd tensor category of
$(G,R)$\nd modules are dualisable.
Suppose chosen duality pairings for $E$ and $E_R$.
Then we have a $G$\nd module $E^{r,s}$ and a $(G,R)$\nd module $(E_R)^{r,s}$
for every $r$ and $s$.
We have canonical embeddings $E \to E_R$ and $E^\vee \to (E_R)^\vee$,
which are compatible with the units and counits of
the chosen duality pairings for $E$ and $E_R$.
They define a canonical embedding
$E^{r,s} \to (E_R)^{r,s}$, which induces an isomorphism of $(G,R)$\nd modules
$(E^{r,s})_R \iso (E_R)^{r,s}$.
Given $u:E^{r,s} \to E^{r',s'}$, we write $u_{R;r,s;r',s'}$
for the unique morphism of $(G,R)$\nd modules
for which the square
\begin{equation}\label{e:daggerdef}
\begin{CD}
(E_R)^{r,s}  @>{u_{R;r,s;r',s'}}>>  (E_R)^{r',s'}   \\
@AAA                              @AAA        \\
E^{r,s}      @>{u}>>             E^{r',s'}
\end{CD}
\end{equation}
commutes, with the vertical arrows the canonical embeddings.
Then $(-)_{R;r,s;r',s'}$ preserves identities and composites,
counits $E^{r,r} \to E^{0,0}$ and $(E_R)^{r,r} \to (E_R)^{0,0}$
and (with the identification $E^{r,0} = E^{\otimes r}$)
commutes with the isomorphisms $\delta_E$ and $\delta_{E_R}$.
For each $r$ and $s$ we have an isomorphism
\begin{equation}\label{e:psidef}
\psi_{r,s}:\Hom_{G,R}((E_R)^{r,s},R) \iso \Hom_G(E^{r,s},R),
\end{equation}
given by composing with the canonical embedding $E^{r,s} \to (E_R)^{r,s}$.
Then
\begin{equation}\label{e:psinat}
\psi_{r,s}(w' \circ u_{R;r,s;r',s'}) = \psi_{r',s'}(w') \circ u
\end{equation}
for every $w':(E_R)^{r',s'} \to R$ and $u:E^{r,s} \to E^{r',s'}$.

Suppose that $G$ is reductive and that $R^G = k$,
so that $R$ has a unique maximal $G$\nd ideal $J$.
Let $N$ be a dualisable $(G,R)$\nd module and $f:R \to N$ be a morphism of $(G,R)$\nd modules
which does not factor through $J N$.
Then $f$ has a left inverse.
Indeed $f^\vee$ does not factor through $J$,
because $f$ is the composite of the unit for $N^\vee$ with $N \otimes_R f^\vee$.
Hence $f^\vee$ is surjective, and there is an $x$ in its source
fixed by $G$ with $f^\vee(x) = 1$.
Thus $f^\vee$ has a unique right inverse $g$ with $g(1) = x$, and
$g^\vee$ is left inverse to $f = f^{\vee \vee}$.

\section{Kimura objects}\label{s:Kimobj}

Let $k$ be a field of characteristic $0$ and $\sC$ be a $k$\nd tensor category
with $\End_\sC(\I) = k$.
An object $L$ of $\sC$ will be called positive (resp.\ negative) if it is dualisable and
$\bigwedge^{r+1} L$ (resp.\ $S^{r+1} L$) is $0$ for some $r$.
An object of $\sC$ will be called a Kimura object if it is the direct sum of a positive
and a negative object of $\sC$.

Let $L$ be a Kimura object of $\sC$.
Then $L = L_0 \oplus L_1$ with $L_0$ positive and $L_1$ negative.
Denote by $l_0$ (resp.\ $l_1$) the least $r$ such that
$\bigwedge^{r+1} L_0$ (resp.\ $S^{r+1} L_1$) is $0$, and let $G$ and $E$ be
as in \eqref{e:Gdef} and \eqref{e:Edef}, and $\rho$ be the central $k$\nd point
of $G$ which acts as $(-1)^i$ on $E_i$.
The goal of this section is to construct a commutative algebra
$R$ in $\REP(G,\rho)$ and an isomorphism
\begin{equation}\label{e:xiiso}
\xi_{r,s}:\Hom_{G,R}((E_R)^{\otimes r},(E_R)^{\otimes s})
\iso \Hom_\sC(L^{\otimes r},L^{\otimes s})
\end{equation}
for every $r$ and $s$, such that the $\xi$ preserve composites and symmetries and are compatible
with $\otimes_R$ and $\otimes$.

Given an object $M$ of $\sC$, write $a_{M,0,r}$ (resp.\ $a_{M,1,r}$) for the image
of the antisymmetrising (resp.\ symmetrising) idempotent of $k[\mathfrak{S}_r]$ under the
$k$\nd homomorphism to $\End(M^{\otimes r})$ that sends $\tau$ in
$\mathfrak{S}_r$ to $M^{\otimes \tau}$.
If $M$ is dualisable of rank $d$, then applying \eqref{e:symcontr} with the $f_j$ the identities
shows that
\[
(r+1)\tr_M(a_{M,i,r+1}) = (d - (-1)^i r)a_{M,i,r}
\]
for $i = 0,1$.
If $M$ is positive (resp.\ negative), it follows that
$d$ (resp.\ $-d$) is the least $r$ for which $\bigwedge^{r+1} M$ (resp.\ $S^{r+1} M$)
is $0$.
Thus $L_i$ has rank $(-1)^i l_i$.

Write $b$ for the automorphism of $L$ that sends $L_i$ to $L_i$ and acts on it as $(-1)^i$.
Then for every $r$, the group $\Gamma_r$ of \eqref{e:semidir} acts on $L^{\otimes r}$
with the action of $(\Z/2)^r$ the $r$th tensor power of the action $i \mapsto b^i$ of
$\Z/2$ on $L$, and the action of $\mathfrak{S}_r$ that given by $\tau \mapsto L^{\otimes \tau}$.
Thus we obtain a homomorphism
\[
\alpha_r:k[\Gamma_r] \to \End_\sC(L^{\otimes r})
\]
of $k$\nd algebras.
If $l_i < r$, then $\alpha_i$ sends the element $x_{i,r}$ of \eqref{e:xir} to the projection onto the
direct summand $\bigwedge^{l_0 + 1}L_0 \otimes_k L^{\otimes (r-l_0-1)}$ when $i = 0$ and
$S^{l_1 + 1}L_1 \otimes_k L^{\otimes (r-l_0-1)}$ when $i = 1$.
Thus both $x_{0,r}$ and $x_{1,r}$ lie in the kernel of $\alpha_r$.
If we write $\beta_r$ for \eqref{e:semidirhom}, it follows by Lemma~\ref{l:GL}~\ref{i:homsurj} that
the kernel of $\alpha_r$ contains that of $\beta_r$.
Hence by Lemma~\ref{l:GL}
there is for each $r$ and $r'$ a unique $k$\nd linear map
\[
\varphi_{r;r'}:\Hom_G(E^{\otimes r},E^{\otimes r'}) \to \Hom_\sC(L^{\otimes r},L^{\otimes r'})
\]
such that
\[
\alpha_r = \varphi_{r;r} \circ \beta_r
\]
for every $r$.
By construction, the $\varphi_{r;r'}$ preserve symmetries, identities and composites,
and they are compatible with $\otimes_k$ and $\otimes$.
Applying \eqref{e:symcontr} $t$ times shows that for $v:E^{\otimes (r+t)} \to E^{\otimes (r'+t)}$ we have
\begin{equation}\label{e:phicontr}
\varphi_{r;r'}(\tr_{E^{\otimes t}}(v)) = \tr_{L^{\otimes t}}(\varphi_{r + t;r' + t}(v)),
\end{equation}
because $\tr(\rho^i) = \tr(b^i) = l_0 - (-1)^il_1$.

For every $r,s$ and $r',s'$, we define a $k$\nd linear map
$\varphi_{r,s;r's'}$ by requiring that the square
\[
\begin{CD}
\Hom_G(E^{r,s},E^{r',s'}) @>{\varphi_{r,s;r's'}}>>  \Hom_\sC(L^{r,s},L^{r',s'}) \\
@A{\delta_{E;r,s;r's'}}AA                             @AA{\delta_{L;r,s;r's'}}A \\
\Hom_G(E^{\otimes (r+s')},E^{\otimes (r'+s)})  @>{\varphi_{r+s';r'+s}}>>
                                        \Hom_\sC(L^{\otimes (r+s')},L^{\otimes (r'+s)})
\end{CD}
\]
commute, with the $\delta$ the isomorphisms of \eqref{e:deltaLdef}.
Then by \eqref{e:deltaLhg} the $\varphi_{r,s;r's'}$ preserve identities,
and by \eqref{e:deltaLcomp} and
\eqref{e:phicontr} they preserve composites.
By \eqref{e:deltaLtens}, they are compatible
with the bilinear products, defined as in \eqref{e:tildedef},
$\widetilde{\otimes}_k$ on $G$\nd homomorphisms
between the $E^{r,s}$ and $\widetilde{\otimes}$ on morphisms between the $L^{r,s}$.
By \eqref{e:deltaLhg},
they send symmetries permuting the factors $E$ or $E^\vee$
of $E^{r,s}$ to the corresponding symmetries of $L^{r,s}$.

We now define as follows a commutative algebra $R$ in  $\REP(G,\rho)$.
Consider the small category $\sL$ whose objects are triples $(r,s,f)$ with $r$ and
$s$ integers $\ge 0$ and $f:L^{r,s} \to \I$, where
a morphism from $(r,s,f)$ to $(r',s',f')$ in $\sL$ is a morphism
$u:E^{r,s} \to E^{r',s'}$ such that
\[
f = f' \circ \varphi_{r,s;r',s'}(u).
\]
Then we define $R$ as the colimit
\[
R = \colim_{(r,s,f) \in \sL} E^{r,s}
\]
in $\REP(G,\rho)$.
Write the colimit injection at $(r,s,f)$ as
\[
i_{(r,s,f)}:E^{r,s} \to R.
\]
We define the unit $\I \to R$ of $R$ as $i_{(0,0,1_\I)}$.
We define the multiplication $R \otimes_k R \to R$ by requiring that
for every  $((r_1,s_1,f_1),(r_2,s_2,f_2))$ in $\sL \times \sL$ the square
\begin{equation}\label{e:musquare}
\begin{CD}
E^{r_1,s_1} \otimes_k E^{r_2,s_2}  @>{\sim}>>    E^{r_1+r_2,s_1+s_2}  \\
@V{i_{(r_1,s_1,f_1)} \otimes_k i_{(r_1,s_1,f_2)}}VV
                      @VV{i_{(r_1+r_2,s_1+s_2,f_1 \mathbin{\widetilde{\otimes}} f_2)}}V \\
R \otimes_k R  @>>>                   R
\end{CD}
\end{equation}
should commute,
where the top isomorphism is that of \eqref{e:tildedef} with $E$ for $L$.
Such an $R \otimes_k R \to R$ exists and is unique because
the left vertical arrows of the squares \eqref{e:musquare} form a colimiting cone
by the fact that $\otimes_k$ preserves colimits, while their top right legs form
a cone by the compatibility of the $\varphi_{r,s;r's'}$ with $\widetilde{\otimes}_k$
and $\widetilde{\otimes}$.
The associativity of the multiplication can be checked by writing
$R \otimes_k R \otimes_k R$ as a colimit over $\sL \times \sL \times \sL$
and using the associativity of $\widetilde{\otimes}$.
The commutativity follows from \eqref{e:tildecom} and the compatibility of
the $\varphi_{r,s;r's'}$  with the symmetries.

Since $G$ is reductive, each $\Hom_G(E^{r,s},-)$ preserves colimits.
Hence
the
\[
\Hom_G(E^{r,s},i_{(r',s',f')}):\Hom_G(E^{r,s},E^{r',s'}) \to \Hom_G(E^{r,s},R).
\]
form a colimiting cone of $k$\nd vector spaces.
Thus for every $r$ and $s$ there is a unique homomorphism
\[
\theta_{r,s}:\Hom_G(E^{r,s},R) \to \Hom_\sC(L^{r,s},\I)
\]
whose composite with $\Hom_G(E^{r,s},i_{(r',s',f')})$ sends $u:E^{r,s} \to E^{r',s'}$ to
\[
f' \circ \varphi_{r,s;r',s'}(u).
\]
Further $\theta_{r,s}$ is an isomorphism, with inverse sending $f:L^{r,s} \to \I$
to $i_{(r,s,f)}$.
Thus every $E^{r,s} \to R$ can be written uniquely in the form $i_{(r,s,f)}$.
It follows that
\begin{equation}\label{e:theta0nat}
\theta_{r,s}(v' \circ u) = \theta_{r',s'}(v') \circ \varphi_{r,s;r',s'}(u)
\end{equation}
for $v':E^{r',s'} \to R$, that $\theta_{0,0}$ sends the identity $k \to R$ of $R$ to $1_{\I}$,
and that
\begin{equation}\label{e:theta0tens}
\theta_{r_1+r_2,s_1+s_2}(v) =
\theta_{r_1,s_1}(v_1) \mathbin{\widetilde{\otimes}} \theta_{r_2,s_2}(v_2)
\end{equation}
for $v_1:E^{r_1,s_1} \to R$ and $v_2:E^{r_2,s_2} \to R$, where $v$ is defined by a diagram of the
form \eqref{e:musquare} with left arrow $v_1 \otimes_k v_2$ and right arrow $v$.

Composing the isomorphisms $\psi_{r,s}$ of \eqref{e:psidef} and $\theta_{r,s}$
gives an isomorphism
\[
\widehat{\theta}_{r,s} = \theta_{r,s} \circ  \psi_{r,s}:
\Hom_{G,R}((E_R)^{r,s},R) \iso \Hom_\sC(L^{r,s},\I).
\]
Then with $u_{R;r,s;r',s'}$ as in \eqref{e:daggerdef},
we have by \eqref{e:psinat} and \eqref{e:theta0nat}
\begin{equation}\label{e:thetanat}
\widehat{\theta}_{r,s}(w' \circ u_{R;r,s;r',s'}) =
\widehat{\theta}_{r',s'}(w') \circ \varphi_{r,s;r',s'}(u)
\end{equation}
for every $w':(E_R)^{r',s'} \to R$ and $u:E^{r,s} \to E^{r',s'}$.
Also $\widehat{\theta}_{0,0}(1_R) = 1_{\I}$, and
\begin{equation}\label{e:thetatens}
\widehat{\theta}_{r_1+r_2,s_1+s_2}(w_1 \mathbin{\widetilde{\otimes}}_R w_2) =
\widehat{\theta}_{r_1,s_1}(w_1) \mathbin{\widetilde{\otimes}} \widehat{\theta}_{r_2,s_2}(w_2)
\end{equation}
for every $w_1:(E_R)^{r_1,s_1} \to R$ and $w_2:(E_R)^{r_2,s_2} \to R$,  by \eqref{e:theta0tens}.

We now define the isomorphism \eqref{e:xiiso} by requiring that the square
\[
\begin{CD}
\Hom_{G,R}((E_R)^{\otimes r},(E_R)^{\otimes s}) @>{\xi_{r,s}}>> \Hom_\sC(L^{\otimes r},L^{\otimes s}) \\
@V{\delta_{E_R;r,s,0,0}}VV                                                   @VV{\delta_{L;r,s,0,0}}V\\
\Hom_{G,R}((E_R)^{r,s},R)      @>{\widehat{\theta}_{r,s}}>>                 \Hom_\sC(L^{r,s},\I)
\end{CD}
\]
commute.
The $\xi$ preserve composites by
\eqref{e:deltacomp}, \eqref{e:thetanat}, \eqref{e:thetatens},
and the fact that $(-)_{R;r,s;r',s'}$ preserves identities and is compatible
with $\delta_E$ and $\delta_{E_R}$.
They are compatible with $\otimes_R$ and $\otimes$ by
\eqref{e:deltaLtens}, where the relevant $\sigma$ and $\sigma'$ reduce to associativities,
and \eqref{e:thetatens}.
They are compatible with the symmetries by \eqref{e:thetanat}
with $w' = 1_R$ and $u$ the composite of $\sigma \otimes_k (E^\vee)^{\otimes r}$
for $\sigma$ a symmetry of $E^{\otimes r}$ with the counit $E^{r,r} \to k$,
using \eqref{e:deltaLf} and the compatibility of $(-)_{R;r,s;r',s'}$ with symmetries,
composites, and counits.

\section{Kimura varieties}\label{s:Kimvar}

We denote by $\sM_{\sim}(F)$ the category of ungraded $\Q$\nd linear motives over $F$
for the equivalence relation ${\sim}$.
It is a $\Q$\nd tensor category.
There is a contravariant functor $h$ from the category $\sV_F$ of
smooth projective varieties over $F$ to  $\sM_{\sim}(F)$,
which sends products in $\sV_F$ to tensor products in  $\sM_{\sim}(F)$.
We then have
\begin{equation}\label{e:HomChow}
\Hom_{\sM_{\sim}(F)}(h(X'),h(X)) = CH(X' \times_F X)_\Q/{\sim},
\end{equation}
and the composite $z \circ z'$ of $z':h(X'') \to h(X')$ with
$z:h(X') \to h(X)$ is given by
\[
z \circ z' = (\pr_{13})_*((\pr_{12})^*(z').(\pr_{23})^*(z)),
\]
where the projections are from $X'' \times_F X' \times_F X$.
Further $h(q)$ for $q:X \to X'$ is the push forward of $1$ in $CH(X)_\Q/{\sim}$
along $X \to X' \times_F X$ with components $q$ and $1_X$.

The images under $h$ of the structural morphism and diagonal of $X$
define on $h(X)$ a canonical structure of commutative
algebra in $\sM_{\sim}(F)$.
With this structure \eqref{e:HomChow} reduces when $X' = \Spec(F)$ to an
equality of algebras
\[
\Hom_{\sM_{\sim}(F)}(\I,h(X)) = CH(X)_\Q/{\sim}.
\]
Also $h(X)$ is canonically autodual: we have canonical duality pairing
\[
(h(X),h(X),\eta_X,\varepsilon_X),
\]
with both $\eta_X$ and $\varepsilon_X$
the class in $CH(X \times_F X)_\Q/{\sim}$ of the diagonal of $X$.
The canonical duality pairing for $h(X \times_F X')$ is the tensor product of those for
$h(X)$ and $h(X')$.
The canonical duality pairings define a transpose $(-)^\vee$ for morphisms $h(X') \to h(X)$,
given by pullback of cycles along
the symmetry interchanging $X$ and $X'$.
For $q:X \to X'$ and $z \in CH(X)_\Q/{\sim}$ and $z' \in CH(X')_\Q/{\sim}$, we have
\begin{equation}\label{e:pull}
q^*(z') = h(q) \circ z'
\end{equation}
and
\begin{equation}\label{e:push}
q_*(z) = h(q)^\vee \circ z.
\end{equation}

A \emph{Kimura variety for ${\sim}$} is a smooth projective variety $X$ over $F$ such
that $h(X)$ is a Kimura object in $\sM_{\sim}(F)$.
If the motive of $X$ in the category of \emph{graded} motives for $\sim$ is a Kimura object,
then $X$ is a Kimura variety for $\sim$.
The converse also holds, as can be seen by factoring out the tensor ideals
of tensor nilpotent morphisms, but this will not be needed.

Let $X$ be a Kimura variety for ${\sim}$.
We may apply the construction of Section~\ref{s:Kimobj} with $k = \Q$, $\sC = \sM_{\sim}(F)$
and $L = h(X)$.
For appropriate $l_0$ and $l_1$, we then have with $G$, $E$ and $\rho$ as in
Section~\ref{s:Kimobj} a commutative algebra $R$ in $\REP(G,\rho)$ and isomorphisms
\[
\xi_{r,s}:\Hom_{G,R}((E_R)^{\otimes r},(E_R)^{\otimes s}) \iso
\Hom_{\sM_{\sim}(F)}(h(X)^{\otimes r},h(X)^{\otimes s})
\]
which are compatible with composites, tensor products, and symmetries.

The homomorphisms of $R$\nd modules $\iota$ and $\mu$ with
respective images under $\xi_{0,1}$ and $\xi_{2,1}$ the unit
and multiplication of $h(X)$ define a structure of
commutative $R$\nd algebra on $E_R$.
Also the homomorphisms $\eta_1$ and
$\varepsilon_1$ with respective images $\eta_X$ and $\varepsilon_X$
under $\xi_{0,2}$ and $\xi_{2,0}$ are the unit and counit
a duality pairing $(E_R,E_R,\eta_1,\varepsilon_1)$ for $E_R$.
We denote by
\[
((E_R)^{\otimes r},(E_R)^{\otimes r},\eta_r,\varepsilon_r)
\]
its $r$th tensor power.
Then $\xi_{0,2r}(\eta_r) = \eta_{X^r}$ and
$\xi_{2r,0}(\varepsilon_r) = \varepsilon_{X^r}$.
For any $(G,R)$\nd homomorphism $f$ from $(E_R)^{\otimes m}$ to $(E_R)^{\otimes l}$
we have
\[
\xi_{l,m}(f^\vee) = \xi_{m,l}(f)^\vee,
\]
where the transpose of $f$ is taken using duality pairings just defined.
Further
\[
\xi_{0,n}:\Hom_{G,R}(R,(E_R)^{\otimes n}) \iso \Hom_{\sM_{\sim}(F)}(\I,h(X)^{\otimes n})
\]
is an isomorphism of $\Q$\nd algebras.
We note that
\[
R^G = \Hom_{G,R}(R,R) = CH(\Spec(F))_\Q/{\sim} = \Q,
\]
by the isomorphism $\xi_{0,0}$.

\section{Proof of Theorem~\ref{t:fin}}\label{s:finproof}

To prove Theorem~\ref{t:fin}, we may suppose that $Z_1$ contains the classes of
the equidimensional components of $X$, and that $Z_n$ contains the homogeneous
components of each of its elements for the grading of $CH(X^n)_\Q/{\sim}$.
Denote by $\sA$ the set of those families $C = ((C_n)^i)_{n,i \in \N}$
with $(C_n)^0$ a $\Q$\nd subalgebra $C_n$ of $CH(X^n)_\Q/{\sim}$
and $((C_n)^i)_{i \in \N}$ a filtration of the algebra $C_n$, such that \ref{i:pullpush},
\ref{i:fin} and \ref{i:nilp} of Theorem~\ref{t:fin} hold.
It is to be shown that there is a $C$ in $\sA$ which is graded, i.e.\
such that $(C_n)^i$ is a graded $\Q$\nd vector subspace of $CH(X^n)_\Q/{\sim}$ for each $n$ and $i$.
For $\lambda \in \Q^*$, define an endomorphism $z \mapsto \lambda * z$ of the algebra
$CH(X^n)_\Q/{\sim}$ by taking $\lambda * z = \lambda^j z$ when $z$ is homogeneous
of degree $j$.
Then the graded subspaces of $CH(X^n)_\Q/{\sim}$ are those that are stable under each
$\lambda * -$.
For each $C$ in $\sA$ we have a $\lambda * C$ in $\sA$ with $((\lambda * C)_n)^i$
the image under $\lambda * -$ of $(C_n)^i$.
Indeed $(\lambda * C)_n$ contains $Z_n$
by the homogeneity assumption on $Z_n$, and
$p_*$ sends $((\lambda * C)_l)^i$ to $((\lambda * C)_m)^i$ for $p$ as in
\ref{i:pullpush}, because $C_l$ contains the classes of the equidimensional
components of each factor $X$ of $X^l$ by the assumption on $Z_1$.
The $C$ in $\sA$ that are graded are then those fixed by each $\lambda * -$.
Now if $\sA$ is non-empty, it has a least element for the ordering of the $C$
by inclusion of the $(C_n)^i$.
Such a least element will be
fixed by the $\lambda * -$, and hence graded.
It will thus suffice to show that $\sA$ is non-empty.

Let $G$, $E$, $\rho$, $R$, $\xi_{r,s}$, $\eta_r$, $\varepsilon_r$, $\iota$ and $\mu$ be as in Section~\ref{s:Kimvar}.
With the identification
\begin{equation}\label{e:CHXn}
\Hom_{\sM_{\sim}(F)}(\I,h(X)^{\otimes n}) = CH(X^n)_\Q/{\sim},
\end{equation}
there exists a finitely generated $G$\nd subalgebra $R'$ of $R$ such that
if we write $\beta_{m,n}$ for the homomorphism \eqref{e:extscal} with $P = E$,
then $(\xi_{0,n})^{-1}(Z_n)$ is contained in the image of $\beta_{0,n}$ for every $n$,
and $\eta_1 = \beta_{0,2}(\eta'{}\!_1)$,
$\varepsilon_1 = \beta_{2,0}(\varepsilon'{}\!_1)$, $\iota = \beta_{0,1}(\iota')$
and $\mu = \beta_{2,1}(\mu')$ for some $\varepsilon'{}\!_1$, $\eta'{}\!_1$,
$\iota'$ and $\mu'$.
We then have duality pairing $(E_{R'},E_{R'},\eta'{}\!_1,\varepsilon'{}\!_1)$ for $E_{R'}$,
and if its $r$th tensor power is
\[
((E_{R'})^{\otimes r},(E_{R'})^{\otimes r},\eta'{}\!_r,\varepsilon'{}\!_r),
\]
we have $\eta_r = \beta_{0,2r}(\eta'{}\!_r)$
and $\varepsilon_r = \beta_{2r,0}(\varepsilon'{}\!_r)$.
Further $\iota'$ and $\mu'$ define a structure of commutative $(G,R')$\nd algebra on $E_{R'}$.
The $\beta_{m,l}$, and hence their composites
\[
\xi'{}\!_{m,l}:\Hom_{G,R'}((E_{R'})^{\otimes m},(E_{R'})^{\otimes n})
\to \Hom_{\sM_{\sim}(F)}(h(X)^{\otimes n},h(X)^{\otimes n})
\]
with the $\xi_{m,n}$,
preserve identities, composition, tensor products, and transposes
defined using the $\eta'{}\!_r$ and $\varepsilon'{}\!_r$.
Further $\xi'{}\!_{0,n}$ is a homomorphism of $\Q$\nd algebras.

We have $R'{}^G = R^G = \Q$.
Thus by Lemma~\ref{l:repfin}~\ref{i:algfin}, the $\Q$\nd algebra
\[
\Hom_{G,R'}(R',(E_{R'})^{\otimes n}) \iso \Hom_G(k,(E_{R'})^{\otimes n})
= ((E_{R'})^{\otimes n})^G
\]
is finite-dimensional for every $n$.
Denote by $J'$ the unique maximal $G$\nd ideal of $R'$.
Then we have for each $n$ a filtration of the $(G,R')$\nd algebra $(E_{R'})^{\otimes n}$
by the $G$\nd ideals
\[
J'{}^r (E_{R'})^{\otimes n},
\]
and hence a filtration of the $\Q$\nd algebra $\Hom_{G,R'}(R',(E_{R'})^{\otimes n})$
by the ideals
\begin{equation}\label{e:homideal}
\Hom_{G,R'}(R',J'{}^r(E_{R'})^{\otimes n}).
\end{equation}
Since \eqref{e:homideal} is isomorphic to $(J'{}^r(E_{R'})^{\otimes n})^G$,
it is $0$ for $r$ large, by Lemma~\ref{l:repfin}~\ref{i:idealcompl}.

We now define an element $C$ of $\sA$ as follows.
With the identification \eqref{e:CHXn}, take for $C_n$ the image of $\xi'{}\!_{0,n}$,
and for $(C_n)^r$ the image under $\xi'{}\!_{0,n}$ of \eqref{e:homideal}.
Then \ref{i:fin} holds.

Let $z = \xi'{}\!_{0,n}(x)$ be an element of $C_n$ which does not lie in $(C_n)^1$.
Then $x$ does not factor through $J'(E_{R'})^{\otimes n}$.
As was seen at the end of Section~\ref{s:dual}, this implies that $x$ has a left
inverse.
Hence $z$ has a left inverse $y:h(X)^{\otimes n} \to \I$.
Identifying $y$ with an element of $CH(X^n)_\Q/{\sim}$, the composite
$y \circ z = 1_\I$ is the push forward of $y.z$ along the structural morphism of $X^n$.
Thus $z$ is not numerically equivalent to $0$.
The first statement of \ref{i:nilp} follows.
The second statement of \ref{i:nilp} follows from the fact that
\eqref{e:homideal} is $0$ for  $r$ large.

Let $p:X^l \to X^m$ be as in \ref{i:pullpush}.
If $p$ is defined by $\nu:[1,m] \to [1,l]$,
then
\[
h(p):h(X)^{\otimes m} \to h(X)^{\otimes l}
\]
is the morphism of commutative algebras in $\sM_{\sim}(F)$ defined by $\nu$.
Thus
\[
h(p) = \xi'{}\!_{m,l}(f)
\]
for $f:(E_{R'})^{\otimes m} \to (E_{R'})^{\otimes l}$ the morphism of
commutative $(G,R')$\nd algebras
defined by $\nu$.
That $p^*$ sends $C_m$ to $C_l$ and respects the filtrations now follows from
\eqref{e:pull} and the compatibility of the $\xi'{}\!_{m,l}$ with composites.
That $p_*$ sends $C_l$ to $C_m$ and respects the filtrations follows from
\eqref{e:push} and the compatibility of the $\xi'{}\!_{m,l}$ with composites
and transposes.
Thus \ref{i:pullpush} holds.

\section{Proof of Theorem~\ref{t:num}}\label{s:numproof}

Let $G$, $E$, $\rho$, $R$, $\xi_{r,s}$, $\eta_r$ and $\varepsilon_r$ be as in
Section~\ref{s:Kimvar}, and suppose that the equivalence relation $\sim$
is numerical equivalence.
We show first that $R$ is $G$\nd simple, i.e.\ has no $G$\nd ideals other
than $0$ and $R$.
Any non-zero $z:h(X)^{\otimes m} \to \I$ has a right inverse
$y$, because $z \circ y$ is the push forward of $z.y$
along the structural isomorphism of $X^m$.
The isomorphisms $\xi$ then show that any non-zero $(E_R)^{\otimes m} \to R$
has a right inverse, and is thus surjective.
Let $J \ne 0$ be a $G$\nd ideal of $R$.
Since $G$ is reductive and $E$ is a faithful representation of $G$,
the category of finite-dimensional representations of $G$ is the pseudo-abelian hull
of its full subcategory with objects the $E^{r,s}$ (\cite{Wat},~3.5).
Thus for some $r,s$ there is a non-zero homomorphism of $G$\nd modules from $E^{r,s}$ to $R$
which factors through $J$.
It defines by the isomorphism \eqref{e:psidef} a non-zero homomorphism of $(G,R)$\nd modules $f$ from
$(E_R)^{r,s}$ to $R$ which also factors through $J$.
Since $(E_R)^{r,s}$ is isomorphic by autoduality of $E_R$ to $(E_R)^{r+s}$,
it follows that $f$ is surjective, so that $J = R$.
Thus $R$ has no $G$ ideals other than $0$ and $R$.

If $R_1$ is the $G$\nd submodule of $R$ on which $\rho$ acts as $-1$, then
the ideal of $R$ generated by $R_1$ is a $G$\nd ideal $\ne R$, because the elements
of $R_1$ have square $0$.
Thus $R_1 = 0$, so that $R$ is commutative as an algebra in $\REP(G)$.
By a theorem of Magid (\cite{Mag}, Theorem~4.5),
the $G$\nd simplicity of $R$ and the fact that $R^G = \Q$ then imply that $\Spec(R)_k$
is isomorphic to $G_k/H$ for some extension $k$ of $\Q$ and closed subgroup $H$ of $G_k$.
Thus $R$ is a finitely generated  $\Q$\nd algebra.
Hence there exists an $n$ such that a set of generators of
$R$ is contained
in the sum of the images of the $G$\nd homomorphisms $E^{r,s} \to R$ for $r+s \le n$.
We may suppose that $n \ge 2$.
We show that $n$ satisfies the requirements of Theorem~\ref{t:num}.

Denote by $U_m$ the $\Q$\nd vector subspace of $\overline{CH}(X^m)_\Q = CH(X^m)_\Q/\sim$
generated by the elements \eqref{e:numgen}, and by
\[
U_{m,l} \subset
\Hom_{G,R}((E_R)^{\otimes m},(E_R)^{\otimes l})
\]
the inverse image of
\[
U_{m+l} \subset
\Hom_{\sM_{\sim}(F)}(h(X)^{\otimes m},h(X)^{\otimes l})
= \overline{CH}(X^{m+l})_\Q
\]
under $\xi_{m,l}$.
The symmetries of $(E_R)^{\otimes m}$ lie in $U_{m,m}$,
because by Proposition~\ref{p:Chowsub} the symmetries of $h(X)^{\otimes m}$ lie
in $U_{2m}$.
Similarly the composite of an element of $U_{m,m'}$ with an element of $U_{m',m''}$ lies in
$U_{m,m''}$, the tensor product of an element of $U_{m,l}$ with an element of
$U_{m',l'}$ lies in $U_{m+m',l+l'}$,
and $\eta_m$ lies in $U_{0,2m}$ and $\varepsilon_m$ lies in $U_{2m,0}$.
Also $U_{m,l}$ coincides with
$\Hom_{G,R}((E_R)^{\otimes m},(E_R)^{\otimes l})$
for $m+l \le n$.

Since $E_R$ is canonically autodual, we may identify $(E_R)^{r,s}$ with $(E_R)^{\otimes (r+s)}$.
The morphism $u_{R;r,s;r',s'}$ of \eqref{e:daggerdef} may then be identified with a morphism
of $R$\nd modules
\[
u_{R;r,s;r',s'}:(E_R)^{\otimes (r+s)} \to (E_R)^{\otimes (r'+s')},
\]
and the isomorphism $\psi_{r,s}$ of \eqref{e:psidef} with an isomorphism
\[
\psi_{r,s}:\Hom_{G,R}((E_R)^{\otimes (r+s)},R) \iso \Hom_G(E^{r,s},R).
\]
Then \eqref{e:psinat} still holds. Also we have a commutative square
\begin{equation}\label{e:psicompat}
\begin{CD}
E^{r_1,s_1} \otimes_\Q E^{r_2,s_2}    @>{\sim}>>   E^{r_1 + r_2,s_1 + s_2}  \\
@V{\psi_{r_1,s_1}(f_1) \otimes_\Q \psi_{r_2,s_2}(f_2)}VV  @VV{\psi_{r_1+r_2,s_1+s_2}(f)}V \\
R \otimes_\Q R   @>>>                             R
\end{CD}
\end{equation}
where the top isomorphism is that of \eqref{e:tildedef} with $E$ for $L$, the bottom
arrow is the multiplication of $R$, and $f$ is the composite of the appropriate
symmetry of $(E_R)^{\otimes (r_1+r_2+s_1+s_2)}$ with $f_1 \otimes_R f_2$.

By Lemma~\ref{l:GL}, a non zero $w:E^{r',0} \to E^{r,0}$ exists only if $r = r'$, when any such $w$
is a composite of symmetries and tensor products of endomorphisms of $E$.
Thus $w_{R;r',0;r,0}$ lies in $U_{r',r}$ for such a $w$, because $n \ge 2$.
Since $(-)_{R;r,s;r',s'}$ commutes with the isomorphisms $\delta$ of \eqref{e:deltaLdef}, it follows
that $w_{R;r',s';r,s}$ lies in $U_{r'+s',r+s}$ for any $w:E^{r',s'} \to E^{r,s}$.

To prove Theorem~\ref{t:num}, write
\[
W_{r,s} = \psi_{r,s}(U_{r+s,0}).
\]
Consider the smallest $G$\nd submodule $R'$ of $R$
such that $a:E^{r,s} \to R$ factors through $R'$ for each $r$, $s$, and $a$ in $W_{r,s}$.
By \eqref{e:psicompat}, $R'$ is a subalgebra of $R$.
Since every $E^{r,s} \to R$ lies in $W_{r,s}$ when $r+s \le n$, the algebra $R'$ contains
a set of generators of $R$.
Hence $R' = R$.
Given $a:E^{r,s} \to R$, there are thus $a_i$ in $W_{r_i,s_i}$ for $i = 1,,2, \dots,t$
such that the image of $a$ lies in the sum of the images of the $a_i$.
By semisimplicity of $\REP(G,\rho)$, it follows that
\[
a = a_1 \circ w_1 + a_2 \circ w_2 + \dots + a_t \circ w_t
\]
for some $w_i$.
Hence by \eqref{e:psinat}, $a$ lies in $W_{r,s}$.
Thus $W_{r,s}= \Hom_G(E^{r,s},R)$ for every $r$ and $s$.
It follows that $U_{m,0} = \Hom_{G,R}((E_R)^{\otimes m},R)$,
and hence $U_m = \overline{CH}(X^m)_\Q$, for every $m$.
This proves Theorem~\ref{t:num}.

\section{Concluding remarks}

Theorem~\ref{t:fin} is easily generalised to the case where instead of cycles on the powers
of a single Kimura variety $X$ for ${\sim}$, we consider also cycles on products of a finite number
of such varieties: it suffices to take for $X$ their disjoint union
and to include in $Z_1$ their fundamental classes.
Similarly in the condition on $X^l \to X^m$ in \ref{i:pullpush}, we may consider
a finite number of morphisms $X^l \to X$ additional to the projections: it suffices
to include in the $Z_i$
the classes of their graphs.
Suppose for example that $X$ is an abelian variety, and
let $\Gamma$ be a finitely generated subgroup of $X(k)$.
Then we may consider in \ref{i:pullpush} pullback and push forward along any morphism
$X^l \to X^m$ which sends the identity of $X(k)^l$ to an element of $\Gamma^m$.

More generally, we can construct a small category $\sV$, an equivalence $T$  from
$\sV$ to the category Kimura varieties over $F$ for ${\sim}$, a filtered family
$(\sV_\lambda)_{\lambda \in \Lambda}$ of (not necessarily full) subcategories $\sV_\lambda$ of $\sV$
with union $\sV$,
and for each $\lambda$ in $\Lambda$ and $V$ in $\sV_\lambda$ a
finite-dimensional graded $\Q$\nd subalgebra $C_\lambda(V)$ of  $CH(T(V))_\Q/{\sim}$
and a filtration $C_\lambda(V)^\bullet$ on $C_\lambda(V)$,
with the following properties.
\begin{enumerate}
\renewcommand{\theenumi}{(\alph{enumi})}
\item
Finite products exist in the $\sV_\lambda$, and
the embeddings $\sV_\lambda \to \sV$ preserve them.
\item
We have $C_\lambda(V)^r \subset C_{\lambda'}(V)^r$ for
$\lambda \le \lambda'$ and $V$ in $\sV_\lambda$, and $CH(T(V))_\Q/{\sim}$ for $V$ in $\sV_\lambda$
is the union of the $C_{\lambda'}(V)$ for $\lambda' \ge \lambda$.
\item
$T(f)^*$ sends $C_\lambda(V')$ into $C_\lambda(V)$
and $T(f)_*$ sends $C_\lambda(V)$ into $C_\lambda(V')$ for $f:V \to V'$ in $\sV_\lambda$,
and $T(f)^*$ and $T(f)_*$ preserve the filtrations.
\item\label{i:surjnilp}
For $V$ in $\sV_\lambda$, the projection from $C_\lambda(V)$ to $\overline{CH}(T(V))_\Q$
is surjective with kernel $C_\lambda(V)^1$, and $C_\lambda(V)^r$ is $0$ for $r$ large.
\end{enumerate}

By applying the usual construction for motives (say ungraded)
to $\sV$ and the $CH(T(V))_\Q/{\sim}$ we obtain a $\Q$\nd tensor category
$\sM$ and a cohomology functor from $\sV$ to $\sM$,
and $T$ defines a fully faithful functor from $\sM$ to $\sM_{\sim}(F)$.
Similarly we obtain from $\sV_\lambda$ and the $C_\lambda(V)$ a (not necessarily full)
$\Q$\nd tensor subcategory $\sM_\lambda$ of $\sM$.
Then each $\sM_\lambda$ has finite-dimensional hom-spaces, and $\sM$ is the filtered
union of the $\sM_\lambda$.
A question involving a finite number of Kimura varieties, a finite number of morphisms
between their products, and a finite number of morphisms between the motives of such
products, thus reduces to a question in some $\sV_\lambda$ and $\sM_\lambda$.
By \ref{i:surjnilp}, the projection from $\sM_\lambda$ to the quotient of  $\sM_{\sim}(F)$
by its unique maximal tensor ideal is full,
with kernel the unique maximal
tensor ideal $\sJ_\lambda$ of $\sM_\lambda$.
Further $\sM_\lambda$ is the limit of the $\sM_\lambda/(\sJ_\lambda)^r$.
Thus we can argue by lifting successively from
the semisimple abelian category $\sM_\lambda/\sJ_\lambda$ to the
$\sM_\lambda/(\sJ_\lambda)^r$.

Theorems~\ref{t:fin} and \ref{t:num} extend easily to the case where the
base field $F$ is replaced by a non-empty connected smooth quasi-projective scheme $S$ over $F$.
For the category of ungraded motives over $S$ we then have $\End(\I) = CH(S)_\Q/{\sim}$,
which is a local $\Q$\nd algebra with residue field $\Q$ and nilpotent
maximal ideal.
All the arguments carry over to this case, provided that Lemma~\ref{l:repfin}
is proved in the more general form where the hypothesis ``$R^G = k$'' is
replaced by ``$R^G$ is a local $k$\nd algebra with residue field $k$''.


\begin{thebibliography}{1}

\bibitem{AndKah}
Y.~Andr{\'e} and B.~Kahn.
\newblock Nilpotence, radicaux et structures mono\"\i dales.
\newblock {\em Rend. Sem. Mat. Univ. Padova}, 108:107--291, 2002.


\bibitem{BAC-1}
N.~Bourbaki.
\newblock {\em \'{E}l\'ements de math\'ematique}.
\newblock Masson, Paris, 1985.
\newblock Alg{\`e}bre commutative. Chapitres 1 {\`a} 4.

\bibitem{FulHar}
W.~Fulton and J.~Harris.
\newblock {\em Representation theory}, volume 129 of {\em Graduate Texts in
  Mathematics}.
\newblock Springer-Verlag, New York, 1991.


\bibitem{Kim}
S.-I. Kimura.
\newblock Chow groups are finite dimensional, in some sense.
\newblock {\em Math. Ann.}, 331(1):173--201, 2005.

\bibitem{Mac}
S.~Mac~Lane.
\newblock {\em Categories for the working mathematician}, volume~5 of {\em
  Graduate Texts in Mathematics}.
\newblock Springer-Verlag, New York, second edition, 1998.

\bibitem{Mag}
A.~R. Magid.
\newblock Equivariant completions of rings with reductive group action.
\newblock {\em J. Pure Appl. Algebra}, 49(1-2):173--185, 1987.

\bibitem{O}
P.~O'Sullivan.
\newblock The structure of certain rigid tensor categories.
\newblock {\em C. R. Math. Acad. Sci. Paris}, 340(8):557--562, 2005.

\bibitem{ShaAlgIV}
I.~R. Shafarevich, editor.
\newblock {\em Algebraic geometry. {IV}}, volume~55 of {\em Encyclopaedia of
  Mathematical Sciences}.
\newblock Springer-Verlag, Berlin, 1994.

\bibitem{Wat}
W.~C. Waterhouse.
\newblock {\em Introduction to affine group schemes}, volume~66 of {\em
  Graduate Texts in Mathematics}.
\newblock Springer-Verlag, New York, 1979.

\end{thebibliography}
\end{document}